\documentclass[a4paper,11pt,reqno]{amsart}
\usepackage{amsmath,amssymb}
\usepackage{enumerate}
\usepackage{ifthen}
\usepackage{graphicx}
\usepackage{enumitem}
\usepackage{calc}
\usepackage{tikz}
\usepackage{enumitem}
\usepackage{tikz-qtree}
\usepackage[hypcap]{caption}
\usetikzlibrary{arrows}
\tikzstyle{block}=[draw opacity=0.7,line width=1.4cm]
\newtheorem{theorem}{Theorem}

\newtheorem{lemma}{Lemma} 
 
\newtheorem{remark}{Remark} 
\newtheorem{assum}{Assumption} 
\allowdisplaybreaks

\begin{document}
\title{Convergence rates of nonlinear inverse problems in Banach spaces: conditional stability and weaker norms}
\maketitle
\begin{center}
\center{ $\text{Gaurav Mittal}$ and    ${\text{Ankik Kumar Giri}}$}
\medskip
{\footnotesize
 % please put the address of the second  and third author
 \center{ Department of Mathematics, Indian Institute of Technology Roorkee, Roorkee, India. }
}\\
\email {gaurav.mittaltwins@gmail.com}
\medskip
\begin{abstract} In this short note, we  formulate the convergence rates of the well known Tikhonov regularization scheme for solving the nonlinear ill-posed problems in Banach spaces. For deriving the convergence rates, we employ the novel smoothness concept of conditional stability estimates in terms of weaker norms. Moreover, we show that our results are applicable on two ill-posed  inverse problems.
\medskip

\hspace{-5mm}\subjclass {}{\textbf{AMS Subject Classifications}: $47$J$06$, $65$J$20$, $47$J$20$, $47$A$52$}\\
\hspace{-5mm}
 
\end{abstract}
\end{center}

\section{Introduction and Preliminaries}

\subsection{Introduction}
\noindent Let $F: D(F) \subset U \to V$ be a nonlinear   operator between the Banach spaces $U$ and $V$, where $D(F)$ is the domain of $F$. Let $\|\cdot\|$ denote the norm of $U$ and $V$. %Let $U^*$ and $V^*$ be the respective duals of the Banach spaces $U$ and $V$ and $\langle \cdot, \cdot \rangle_{U^*, U}$ represents the dual pairing between $U^*$ and $U$.
 We study  the  nonlinear ill-posed problems governed  by the operator equation  \begin{equation}\label{1}
F(u) = v, \quad u\in D(F), v \in V.
\end{equation}  Since, in practice, the exact data is not always available, we look for the approximate solutions of (\ref{1}). In general,   variational regularization (in particular,  Tikhonov regularization) \cite{Engl}, regularization with sparsity constraints \cite{Scherzer},   regularization via iterative methods \cite{Kalten, Schuster} etc.,  are used for finding the stable approximations of the exact solution of (\ref{1}). When $U$ and $V$ are Hilbert spaces, a widely used regularization method is  Tikhonov regularization,  which involves the minimization of the  functional $$\|F(u)-v'\|^2+\alpha\|u-u_0\|^2, \ u\in D(F),\ \alpha > 0,$$  where $v'$ is some noisy approximation of $v$ and $u_0\in U$ is an initial guess of the exact solution of (\ref{1}).  Because of its tendency to smooth the solutions, Tikhonov regularization does not yield satisfactory output in Hilbert spaces,  especially if there are jumps or sparsity in the structure of exact solution. So, in the recent years other types of Tikhonov regularization have been explored in  Banach spaces, see \cite{Engl, Scherzer, Schuster}.  More specifically, the following generalized Tikhonov regularization is widely considered in the Banach spaces \begin{equation}\label{2}
T_{\alpha}(u, \delta) := \|F(u)-v^{\delta}\|^p + \alpha R(u), \ u\in D(F),
\end{equation} with $1 \leq p < \infty$, $R: U \to [0, \infty]$ is a  stabilizing functional which is convex and proper, i.e.   $D(R)\neq \emptyset$,  where $D(R)=\{u\in U:R(u)<\infty\}$ is the domain of $R$,  $\alpha$ is the regularization parameter and $v^{\delta}$ is some noisy approximation of $v\in V$ satisfying \begin{equation}\label{3}
\|v^{\delta}-v\| \leq \delta.
\end{equation} 
Typically, for the minimizers of Tikhonov regularization, the following four issues, i.e.,  existence, stability,  convergence  to the exact solution and  convergence rates  are of the most interest \cite{Scherzer}.   For the Tikhonov minimizers of (\ref{2}),  these four issues are well addressed in \cite{Scherzer}. 
  The determination of     convergence rates have a very long tradition in inverse problems (see, e.g.,  \cite{Baku1, Baku2, Engl, Kabanikhin, Kalten, Neubauer, Werner}). In order to derive the convergence rates, some kind of smoothness of the exact solution is utilized in terms of the source conditions or variational inequalities \cite{Engl, Scherzer, Schuster}. One of the recently developed smoothness concepts is to study the convergence rates in Hilbert as well as Banach spaces is that of conditional stability estimates \cite{Mittal1, Mittal2, Mittal3, Schuster, Werner}.
 In this paper, we are peculiarly interested in the convergence rates at which minimizers of (\ref{2}) with $R(u)=\|u-u_0\|^2$ for an initial guess $u_0$ of the exact solution, converges to the exact solution of (\ref{1}). The highlight of our work is that we employ the following  novel smoothness concept of conditional stability in weaker norms to derive the convergence rates:\begin{equation}\label{4}
 \|u-u^{\dagger}\|_{\eta}\leq C\ \eth(\|F(u)-F(u^{\dagger})\|),\ C>0,  
 \end{equation} where $\|\cdot\|_{\eta}$ is a norm weaker than the strong norm $\|\cdot\|$ (because it is induced from a topology weaker than the norm topology), $\eth$ is an index function, i.e.  $\eth:[0,\infty)  \to [0, \infty)$ is such that  $(i)$ $\eth(0)=0$, $(ii)$ $\eth$ is strictly monotonically increasing and continuous;  and $u^{\dagger}$ is a $R$-minimizing  solution of (\ref{1}).  
   We recall that an element $u^{\dagger} \in D$, where $D:= D(F)\cap D(R)$, is called a $R$-minimizing solution of (\ref{1})  if it satisfies (\ref{1}) and  minimizes the functional $R$, i.e., $R(u^{\dagger})= \min\{R(u):  u \in D\ \text{such that}\ F(u) = v \}.$ We remark that if the norm $\|\cdot\|_{\eta}$ is replaced with the usual norm $\|\cdot\|$ in (\ref{4}), then it becomes the usual conditional stability estimate \cite{Schuster}. 
\subsection{Preliminaries and assumptions} 
  In this section, we revisit some well known definitions and results required in this paper. Moreover, we also discuss the various assumptions that are needed to prove our main result.\newline
\textbf{Bregman distance}:
Let $R: U \to \mathbb{R}^+ \cup \{0, \infty\}$  be a   convex and proper functional having subdifferential $\partial R$ and  $u\in U$ be such that $\zeta \in \partial R(u) \subset U^*$. Here $U^*$ is the dual space of $U$.  Then,  with respect to $R$, the Bregman distance $D_{\zeta}( u^*, u)$ for any element $u^* \in U$ from $u$ is given by    \begin{equation} \label{5}
 D_{\zeta}( u^*, u) := R(u^*) - R(u) - \langle \zeta, u^*-u \rangle_{U^*, U}, \end{equation} where $\langle \zeta, u^*-u \rangle_{U^*, U}$ denotes the dual pairing between $U^*$ and $U$. By definition of Bregman distance, it is clear that $D_{\zeta}(\cdot, u)$ is defined at $u$ only if  $\partial R(u) \neq \emptyset$. \newline For variational regularization methods, it is more appropriate to show the convergence of minimizers of (\ref{2})  in  Banach spaces  via Bregman distances instead of Ljapunov functional or conventional norm  based functionals \cite{Osher}.
\begin{remark}
If we consider $R(u) = \frac{1}{2}\|u\|^2$, then $(\ref{5})$ gives  $D_{\zeta}( u^*, u) = \frac{1}{2}\|u^*-u\|^2$. Further, one can observe that   $D_{\zeta}(\cdot, \cdot)$ is similar to a metric but, in general, it does not satisfy the symmetric property and the triangular inequality.
\end{remark}
\noindent To assure the well-posedness, convergence (to the exact solution), and stability of the   Tikhonov regularized solutions of (\ref{2}),  the following assumptions are required.
\begin{assum}
$(1)$ The Banach spaces $U$ and $V$ are associated with the topologies $O_U$ and $O_V$, respectively,  and these topologies are coarser than the norm topology.\newline 
$(2)$ The norm $\|\cdot\|_V$ is sequentially lower semi-continuous with respect to  $O_V$. \newline
$(3)$ The functional $R: U \to [0, \infty]$ is  sequentially lower semi-continuous with respect to $O_U$  as well as convex.\newline
$(4)$ $D= D(F) \cap D(R) \neq \emptyset$.\newline
$(5)$ The level sets $M_{\alpha}(K)$, for every $\alpha > 0$ and $K > 0$, are  sequentially pre-compact and closed with respect to $O_{U}$, where  $M_{\alpha}(K):= \{u \in D : T_{\alpha}(u, 0) \leq K\}.$\newline
$(6)$ The restriction of $F$ on $ M_{\alpha}(K)$ is  sequentially continuous with respect to $O_U$ and $O_V$, for every $\alpha > 0$ and $K>0$. 
\end{assum}\noindent
The next lemma is a well-known result on the existence, stability and convergence of  Tikhonov minimizers of (\ref{2}), provided  Assumption $1$ is satisfied. 
\begin{lemma}\cite[Chapter 3]{Scherzer}
Let Assumption $1$ hold. Then:\newline
$(1)$ the minimizer of the  Tikhonov functional $(\ref{2})$ exists for any $\alpha > 0$ and $v^{\delta} \in V$. \newline
$(2)$ the minimizers of $(\ref{2})$ are stable in the sense that if $\{v^k\}$ is a sequence converges to $v^{\delta} \in V$, with respect to the norm topology,  then the each corresponding sequence $\{u^k\}$ of  minimizers of $(\ref{2})$, with $v^{\delta}$ replaced by $v^k$, has a convergent subsequence $\{u^{k'}\}$ with respect to $O_U$. The limit of the subsequence $\{u^{k'}\}$ is a minimizer of $(\ref{2})$.\newline 
$(3)$ if there exists  a solution of $(\ref{1})$  in $D$ and $\alpha  :(0, \infty) \to (0, \infty)$ fulfills $
\alpha(\delta) \to 0$  and   $(\alpha(\delta))^{-1}\delta^p \to 0$  as $\delta \to 0$, then the minimizers of the  Tikhonov functional $(\ref{2})$ converges  to the solution of $(\ref{1})$.
\end{lemma}
\section{Convergence rates in terms of weaker norms and examples}
In this section, we formulate the convergence rates for Tikhonov regularization scheme (\ref{2}) by incorporating the stability estimates (\ref{4}) in the form of Theorem 1.\subsection{Main result} 
 \begin{theorem} Let $F:D(F)\subset U \to V$ be a nonlinear operator between the Banach spaces $U$ and $V$. Moreover, let the following assumptions hold:
\newline
$(1)$ $u^{\dagger}$ is a $R$-minimizing solution of $(\ref{1})$. \newline
$(2)$ Assumption $1$ holds with $R(u) = \|u-u_0\|^p$. \newline
$(3)$ $u_0$ is in some neighborhood of $u^{\dagger}$, i.e., there exists some $K_1 > 0$ such that \begin{equation}\label{6}
 \|u_0-u^{\dagger}\| \leq K_1. \end{equation}
$(4)$  {$F:   D(F)\subset U \to V$ satisfies the following conditional stability estimate: 
 \begin{equation}\label{7}
 \|u_1-u_2\|_{\eta} \leq C\ \eth(\|F(u_1)-F(u_2)\|),
 \end{equation} for an index function $\eth$}, provided $\|u_1- u_2\| \leq  M$,   where $M$ and $C$ are positive  constants and $\|\cdot\|_{\eta}$ is an appropriate norm which induces a topology coarser than  the  norm topology on $U$. \newline
Further, let $\alpha = \alpha(\delta)$ be such that $0<\alpha(\delta) \leq \alpha_{\max}$,  $\frac{\delta^p}{\alpha} \leq c$ for  $c >0$, $M$ and $K_1$ satisfy \begin{equation}\label{8}
(c+K_1^p)^{\frac{1}{p}} +K_1 \leq M, \ \ K_1^p\leq \frac{K_2}{\alpha_{\max}}
\end{equation} for some $K_2>0$. Let $u_{\alpha}^{\delta}$ be the minimizer of   $(\ref{2})$. Then, for $\alpha(\delta)\sim \delta^{p- \epsilon}$,    we have {\begin{equation} \label{9}
 \|u_{\alpha}^{\delta}- u^{\dagger}\|_{\eta} = O\big(\eth\big(\delta^{\frac{p-\epsilon}{p}}\big)\big), \ \ \text{for}\ \delta \to 0, \ \text{where}\   0 < \epsilon < p. 
  \end{equation}} \end{theorem}
 \begin{proof}
From the definition of $u_{\alpha}^{\delta}$, (\ref{3}) and (\ref{6}),  we have that  \begin{equation}\label{9e}
 \|F(u_{\alpha}^{\delta})-v^{\delta}\|^p + \alpha\|u_{\alpha}^{\delta}-u_0\|^p \leq \|F(u^{\dagger})-v^{\delta}\|^p + \alpha\|u^{\dagger}-u_0\|^p
 \leq \delta^p + \alpha K_1^p.
  \end{equation}
  Consequently, the non-negativity of the norm gives  
   \begin{equation}\label{10}
  \|u_{\alpha}^{\delta}-u_0\|^p \leq \frac{\delta^p}{\alpha}+K_1^p \leq c+K_1^p.  \end{equation} 
  By using (\ref{6}), (\ref{8}) and (\ref{10}) along with the triangular inequality, we derive that \begin{equation*}
 \|u_{\alpha}^{\delta}-u^{\dagger}\| \leq \|u_{\alpha}^{\delta}-u_0\|+ \|u_0-u^{\dagger}\| \leq M.
\end{equation*}   
    The last inequality implies that (\ref{7}) is applicable if we take $u_1=u_{\alpha}^{\delta}$ and $u_2=u^{\dagger}$.  {Hence,   (\ref{7}) and (\ref{3})    along with the definition of index function yield \begin{equation}\label{11}
    \|u_{\alpha}^{\delta}-u^{\dagger}\|_{\eta} \leq C\ \eth(\|F(u_{\alpha}^{\delta})-F(u^{\dagger})\|) \leq C \ \eth(\big(\|F(u_{\alpha}^{\delta})-v^{\delta}\|+\delta\big).
    \end{equation}}  To this end, we deduce from (\ref{9e}) and (\ref{8}) that \begin{equation*}
\|F(u_{\alpha}^{\delta})-v^{\delta}\|^p    \leq \|F(u^{\dagger})-v^{\delta}\|^p + \alpha K_1^p
\leq  \alpha \bigg(  
\frac{\delta^p}{\alpha} +   \frac{K_2}{\alpha_{\max}}\bigg).
\end{equation*} 
    Plugging this estimate in (\ref{11}) to obtain  {\begin{equation*}
\|u_{\alpha}^{\delta}-u^{\dagger}\|_{\eta}  \leq C\  \eth\bigg(\alpha^{\frac{1}{p}}\bigg(  
\frac{\delta^p}{\alpha} +   \frac{K_2}{\alpha_{\max}}\bigg)^{\frac{1}{p}} + \delta\bigg). 
\end{equation*} Further, as $\frac{\delta^p}{\alpha} \leq c$,  we get \begin{equation*}
\|u_{\alpha}^{\delta}-u^{\dagger}\|_{\eta}   \leq C\  \eth\bigg(\alpha^{\frac{1}{p}}\bigg(  
c +   \frac{K_2}{\alpha_{\max}}\bigg)^{\frac{1}{p}} + \delta\bigg). \end{equation*}
Finally,  for the a-priori choice of $\alpha = \alpha(\delta) \sim \delta^{p-\epsilon}$ with $0 < \epsilon < p$, we get {\begin{equation*} 
 \|u_{\alpha}^{\delta}- u^{\dagger}\|_{\eta} = O\big(\eth\big(\delta^{\frac{p-\epsilon}{p}}\big)\big) \ \ \text{for}\ \delta \to 0, 
  \end{equation*}}} where   $0 < \epsilon < p$. This completes the proof. \end{proof}\noindent
\begin{remark} It is worth to mention that  the regularization is only used in Theorem $1$ to constrain the  minimizers of  $(\ref{2})$ to a set on which the conditional stability estimate $(\ref{7})$ hold.\end{remark}\subsection{Examples}\noindent
 Next, we discuss two examples of ill-posed inverse problem on which the result of Theorem $1$ can be applied. In the first example, we consider the inverse problem which involves the   determination of a potential function from the knowledge of Neumann to Dirichlet map for the wave equation (cf. \cite{Bao2}). In the second example, we discuss an ill-posed classical inverse scattering problem.
\newline
\textbf{Example 1} (\textbf{Inverse problem related with a wave equation}). 
We consider the wave equation \begin{equation}\label{13}\begin{cases}\begin{split}
u_{tt}-\Delta u+qu = 0 \qquad \text{for all} \ (x, t) \ \in \Omega \times (0, T),
\\ u = u_t = 0  \qquad\qquad\text{for all} \ x  \in \Omega \ \text{and} \  t = 0,\hspace{4mm}  \\ \frac{\partial u}{\partial \nu} = f \qquad\qquad \text{for all} \ (x, t) \ \in \partial \Omega \times (0, T),\end{split}\end{cases}
\end{equation}where $\Omega$ is a  bounded open set in $\mathbb{R}^n$ having a smooth boundary $\partial \Omega$  for $n \geq 2$, $T > \ $diameter$(\Omega)$, $q=q(x)$ is the potential function and $\nu$ is the unit outward normal.   The associated Neumann to Dirichlet map $\Lambda_q$ with $(\ref{13})$  is defined as follows: \begin{equation*}
\Lambda_q : f \in L^2(\partial \Omega \times (0, T))   \to u|_{\partial \Omega \times (0, T)} \in H^1(\partial \Omega \times (0, T)).
\end{equation*}
Let us consider the inverse problem  of determining  the potential function $q = q(x)$ from $\Lambda_q$. 
Bao and Yun \cite{Bao2}, established the following  conditional stability estimate  for  the  above-mentioned inverse problem. \newline
\textbf{Conditional stability estimate}:
Suppose that $\|q_1-q_2\|_{H^{\beta}(\mathbb{R}^n)} \leq M^{\beta}$ for some $\beta >\frac{n}{2}+1$, $M>0$ and $\|q_i\|_{L^{\infty}(\mathbb{R}^n)} \leq M$, $i=1, 2$. Then, we have \begin{equation}\label{14}
  \|q_1-q_2\|_{L^{\infty}(\Omega)}\leq C \|\Lambda_{q_1}-\Lambda_{q_2}\|_*^{1-\epsilon},
\end{equation}
where $0<\epsilon <1$ and $C>0$. Here $\|\cdot\|_*$ denotes the operator norm from $L^2(\partial \Omega \times (0, T))$ to $H^1(\partial \Omega \times (0, T))$ and $H^{\beta}(\mathbb{R}^n)$ is the  
Sobolev space of order $\beta$.\newline \noindent  For our convenience, we write the  inverse problem under consideration as\begin{equation}\label{15}
 F: q \in L^{\infty}(\Omega) \mapsto \Lambda_q \in B\big(L^2 (\partial \Omega \times (0, T)), H^1(\partial \Omega \times (0, T))\big),  \end{equation}  where $B\big(L^2 (\partial \Omega \times (0, T)), H^1(\partial \Omega \times (0, T))\big)$ is the  Banach space of all bounded linear maps from $L^2 (\partial \Omega \times (0, T))$ to  $H^1(\partial \Omega \times (0, T))$. 
 Let us assume that there exists a $R$-minimizing  solution  $q^{\dagger}$ of (\ref{15}) corresponding to the wave equation (\ref{13}), i.e., $F(q^{\dagger}) = \Lambda_q$.  Also let the initial approximation $q_0$ of the exact solution be such that  $\|q^{\dagger}-q_0\|_{H^{\beta}} \leq K_1$ for some $K_1 > 0$.  {It can be easily seen   that assumption $4$ of Theorem $1$ is satisfied via the conditional stability estimate (\ref{14})  with the index function $\eth(s)=s^{1-\epsilon}$, $0<\epsilon<1$}. Therefore,  Theorem $1$ provides the convergence rates 
 $$\|q_{\alpha}^{\delta}- q^{\dagger}\|_{L^{\infty}(\Omega)}  = O\bigg(\delta^{\frac{(1-\epsilon)(p-\epsilon_1)}{p}}\bigg),$$ provided the regularization parameter is such that $\alpha(\delta)\sim \delta^{p- \epsilon_1}$, where $\epsilon_1<p$, $q_{\alpha}^{\delta}$ is the  minimizer of the following Tikhonov functional $\|F(q)- \Lambda_q^{\delta}\|_*^p+ \alpha \|q-q_0\|_{H^{\beta}}^p$ and Assumption $1$ holds.\newline  
\textbf{Example 2} (\textbf{Inverse scattering problem}).  The classical inverse scattering problem
is concerned with the recovery of  refractive index of a medium  in the presense of near or far field measurements of scattered time-harmonic acoustic waves \cite{Hohage}. The forward problem can be defined as follows:  Given one or several wave$($s$)$ $u^i$ which solves the Helmholtz equation $\Delta u^i+k^2u^i=0$ and a refractive index $\eta=1-f$, determine the total field$($s$)$ $u=u^i+u^s$ such that \begin{equation}\label{16}
 \Delta u+k^2\eta u=0, \ \text{in}\ \mathbb{R}^3\end{equation} \begin{equation}\label{17}
 \frac{\partial u^s}{\partial z}-\iota ku^s=O\bigg(\frac{1}{z^2}\bigg),\  \text{as}\ |x|=z\to\infty. \end{equation} The condition $(\ref{17})$ is known as the  Sommerfeld radiation condition.  
 The corresponding inverse problem is the reconstruction of the refractive index $\eta=1-f$ from the far field data $u^{\infty}$. Note that every solution of $(\ref{16})$ satisfying $(\ref{17})$ has the following asymptotic behavior \begin{equation*}
 u(x)=u^i(x)+\frac{e^{\iota kz}}{z}\bigg(u^{\infty}(\hat{x})+O\bigg(\frac{1}{z^2}\bigg)\bigg), \ \text{as}\ |x|=z\to\infty,
 \end{equation*}
uniformly for all directions $\hat{x}=\frac{x}{t}$ is in the set $\mathbb{S}=\{x\in \mathbb{R}^3: |x|=1\}$.
Mathematically, the forward operator $F$ is the mapping  $$F:D(F)\subset L^{\infty}(\mathbb{R}^3)\to L^2(\mathbb{S}\times \mathbb{S}): F(f)=u^{\infty}.$$ 
In  \cite[Corollary 2.5]{Hohage}, it has been shown that, under certain conditions,  the inverse problem under consideration fulfills the following conditional stability estimate $$\|f_1-f_2\|_{H^m}\leq K\big(\ln(3+\|F(f_1)-F(f_2)\|^{-2}_{L^2(\mathbb{S}\times \mathbb{S})})\big)^{-\mu\theta},\ \ 0<\theta<1,$$
where $f_1$ and $f_2$ are such that  $$\|f_1-f_2\|_{H^s}\leq C_s,\ \   s\neq 2m+\frac{3}{2}$$ for  $\frac{3}{2}<m<s,$ $K>0, C_s>0$  and $H^m$ is the Sobolev space. Clearly, an estimate  of the form $(\ref{4})$ holds with the index function  $\eth(t)=\big(\ln(3+t^{-2})\big)^{-\mu\theta}$, where  $ 0<\theta<1$ and  $\mu\leq 1$ is some constant. This means the results of Theorem 1 are also applicable on this inverse problem.
\subsection{Comparison Analysis}
In this section, we compare the convergence rates results of  Tikhonov regularization  method obtained in Theorem 1 with other notable works in this direction.  \newline
\subsubsection{Comparison with the results of \cite{Schuster}}
Schuster et al. \cite{Schuster} considered the following Tikhonov regularization method:\begin{equation*}
\frac{1}{p} \|F(u)-v^{\delta}\|^p + \alpha \frac{1}{q}\|u\|_{W}^q, \ u\in D(F)\cap W\subset U,
\end{equation*} where $W$ is a dense and continuously embedded subset of $U$ and $1<p, q<\infty$. The other symbols used above carry the same meaning as in (\ref{2}). Further, the conditional stability estimate considered in \cite{Schuster} can be written as $$\|u-u^{\dagger}\|\leq C\ \eth(\|F(u)-F(u^{\dagger})\|),\ C>0, u, u^{\dagger}\in M_R,$$ where $M_R=\{u\in D(F)\cap W:\|u\|_W\leq R\}$ and $\eth$ is a concave index function. Then, under certain conditions, the following convergence rates have been established
 \begin{equation}\label{18} 
 \|u_{\alpha}^{\delta}- u^{\dagger}\| = O\big(\eth(\delta)\big) \ \ \text{for}\ \alpha=\alpha(\delta)\sim \delta^p,
  \end{equation}
where $u^{\dagger}\in D(F)\cap W$. The convergence rates obtained in Theorem $1$ are \begin{equation} \label{19} 
 \|u_{\alpha}^{\delta}- u^{\dagger}\|_{\eta} = O\big(\eth\big(\delta^{\frac{p-\epsilon}{p}}\big)\big) \ \ \text{for}\ \alpha=\alpha(\delta)\sim \delta^{p-\epsilon}, 
  \end{equation} where   $0 < \epsilon < p$. Note that if $\epsilon\to 0$, then the two convergence rates (\ref{18}) and (\ref{19}) are of the same order. This shows that the convergence rates of Tikhonov regularization in terms of weaker norms are of the same order as in the case of stronger norm. 

\subsubsection{Comparison with the results of \cite{Werner}}
Werner et al. \cite{Werner} considered the following Tikhonov regularization method:\begin{equation*}
S(F(u), v^{\delta}) + \alpha \|u\|_s^2, \ u\in D(F),
\end{equation*} where $U$ and $V$ are real Hilbert spaces, $s>0$ and $S(\cdot, \cdot)$ is a data fidelity term. 
 Further, the conditional stability estimate considered in \cite{Werner}  is  of the form $$\|u-u^{\dagger}\|_{-a}\leq C\ \eth(\|F(u)-F(u^{\dagger})\|),\ C>0, u\in Q,$$ where $a\geq 0$, $0\leq s< x\leq 2s+a, -a<s$ for $s\in \mathbb{R}$,  $Q$ is a subset of a Hilbert scale  generated through a linear self-adjoint strictly positive operator (see \cite{Engl}) and $\eth$ is a concave index function.
  Then, under certain conditions, the following convergence rates have been established
 \begin{equation}\label{20} 
 \|u_{\alpha}^{\delta}- u^{\dagger}\|_s = O\big((\eth(\delta))^{\frac{x-s}{a+x}}\big) \ \ \text{for}\ -\frac1{\alpha}\in \partial\bigg(-\big(\eth(\delta)\big)^{\frac{2(x-s)}{a+x}}\bigg),
  \end{equation}
where $u\in Q$, $u^{\dagger}$ uniquely solves (\ref{1}) and $\partial h(x)$ denotes the subdifferential of $h$ \cite{Scherzer}. Now, if $a=0$, then that means the convergence rates (\ref{20}) holds in a subset of $U$ and they are of the order of $O\big((\eth(\delta))^{\frac{x-s}{x}}\big)$. Note that the choice of $\alpha$ mentioned in (\ref{20}) fulfills $(3)$ of Lemma $1$ with $p=2$. Therefore, as $s\to 0$, these convergence rates are of the order of $O\big(\eth(\delta)\big)$, i.e., they are of the same order as the convergence rates mentioned in (\ref{18}) and (\ref{19}).
\section{Discussion}

We have obtained the convergence rates of  Tikhonov regularization scheme for non-linear ill-posed problems in Banach spaces by incorporating the  novel smoothness concept of conditional stability estimates  in weaker norms.  Moreover, the theory presented in Theorem  $1$ is complemented  with two concrete examples. Also, we have compared our convergence rates results with several other notable works that engaged conditional stability estimates to derive the convergence rates of Tikhonov regularization and  shown that they are of the same order.

\noindent One of the important future questions related to this work is to determine the optimal rates of convergence that can be achieved via conditional stability estimates. In this direction, the notable works of \cite{Egger, Neubauer} can be used as a reference.

\bibliographystyle{plain}

\end{document}